\documentclass[12pt,reqno]{amsart}
\usepackage{indentfirst, amssymb,amsmath,amsthm, mathrsfs,cite,graphicx,float}
\newtheorem{thm}{Theorem}[section]

\newtheorem{cor}{Corollary}[section]

\newtheorem{rem}{Remark}[section]
\newtheorem{ex}{Example}[section]
\newtheorem{defi}{Definition}[section]

\newtheorem*{thmA}{Theorem A}

\newtheorem{open problem}{Open problem}[section]
\newcommand{\pa}{\partial}
\newcommand{\ol}{\overline}
\newcommand{\be}{\begin{equation}}
\newcommand{\ee}{\end{equation}}
\newcommand{\bs}{\begin{small}}
\newcommand{\es}{\end{small}}
\newcommand{\beas}{\begin{eqnarray*}}
\newcommand{\eeas}{\end{eqnarray*}}
\newcommand{\bea}{\begin{eqnarray}}
\newcommand{\eea}{\end{eqnarray}}
\renewcommand{\epsilon}{\varepsilon}
\renewcommand{\phi}{\varphi}
\numberwithin{equation}{section}
\begin{document}
\title[Pre-Schwarzian and Schwarzian derivatives]{On the pre-Schwarzian and Schwarzian derivatives of log-harmonic mappings}
\author[R. Biswas and R. Mandal]{Raju Biswas and Rajib Mandal}
\date{}
\address{Raju Biswas, Department of Mathematics, Raiganj University, Raiganj, West Bengal-733134, India.}
\email{rajubiswasjanu02@gmail.com}
\address{Rajib Mandal, Department of Mathematics, Raiganj University, Raiganj, West Bengal-733134, India.}
\email{rajibmathresearch@gmail.com}
\let\thefootnote\relax
\footnotetext{2020 Mathematics Subject Classification: 31A05, 30C35, 30C45, 30C62.}
\footnotetext{Key words and phrases: Pre-Schwarzian and Schwarzian derivatives, log-harmonic mappings, univalence criterion.}
\begin{abstract}
In this paper, we introduce definitions of the pre-Schwarzian and the Schwarzian derivatives for any locally univalent log-harmonic mappings defined in the unit disk $\mathbb{D}=\{z\in\mathbb{C}: |z|<1\}$. We explore the properties and applications of these concepts in the context of geometric function theory, and we also establish a necessary and sufficient condition for a non-vanishing log-harmonic mapping having a finite pre-Schwarzian norm. Additionally, we establish a relationship between the pre-Schwarzian norm of a non-vanishing log-harmonic mapping and that of a certain analytic function in $\mathbb{D}$.
\end{abstract}
\maketitle
\section{Introduction}
\noindent In geometric function theory, the pre-Schwarzian and Schwarzian derivatives of locally univalent and analytic functions are currently a prominent technique for studying the geometric properties of these mappings. One of the potential applications of this approach can be used to determine the conditions for global univalence. The elegance of these derivatives has led to the extension of the theory to complex-valued harmonic mappings (see \cite{CDO2003,HM2015}) and non-vanishing log-harmonic mappings (see \cite{1LP2018, BHPV2022}). The theory of log-harmonic mapping has become an exciting field of research in the last few years. 
The natural follow-up question is,``can these derivatives be defined for any log-harmonic mapping, including those that vanish at the origin?''\\[2mm]
\indent In this paper, we answer this question affirmatively by introducing definitions of the pre-Schwarzian and Schwarzian derivatives for any locally univalent log-harmonic mapping defined in the unit disk $\mathbb{D}$. Our approach provides a unified framework that encompasses both vanishing and non-vanishing cases. We then explore their properties and applications, particularly in the context of geometric function theory. By establishing connections with existing results, we aim to provide a more profound understanding of these derivatives and their implications for the field.\\[2mm]
\indent An analytic function $f(z)$ in a domain $\Omega$ is said to be locally univalent if for each $z_0\in\Omega$, there exists a neighborhood $N(z_0)$ of 
$z_0$ such that $f(z)$ is univalent in $N(z_0)$. The Jacobian of a complex-valued function $f=u+i v$ is defined by $J_f(z)=|f_z|^2 - |f_{\overline{z}}|^2$. 
It is well known that the non-vanishing of the Jacobian is necessary and sufficient conditions for local univalence of analytic mappings (see \cite[Chapter 1]{D1983}). 
Let $\mathcal{S}$ denotes the class of all analytic and univalent function $f$ in $\mathbb{D}$ with normalization $f(0)=f'(0)-1=0$.  An analytic function $f$ defined in $\mathbb{D}$ is called a Bloch function (see \cite{ACP1974, P1970}) if
\beas \beta_f=\sup_{z\in\mathbb{D}}(1-|z|^2)|f'(z)|<\infty.\eeas
Let $\mathcal{B}$ be the class of all analytic functions $\omega:\mathbb{D}\rightarrow\mathbb{D}$ and $\mathcal{B}_0=\{\omega\in\mathcal{B} : \omega(0)=0\}$. 
Functions in $\mathcal{B}_0$ are called Schwarz function. According to Schwarz's lemma, if $\omega\in\mathcal{B}_0$, then $|\omega(z)|\leq |z|$ and $|\omega'(0)|\leq 1$. 
Strict inequality holds in both estimates unless $\omega(z)=e^{i\theta}z$, $\theta\in\mathbb{R}$. A sharpened form of the Schwarz lemma, known as the Schwarz-Pick lemma, 
gives the estimate  $|\omega'(z)|\leq (1-|\omega(z)|^2)/(1-|z|^2)$ for $z\in\mathbb{D}$ and $\omega\in\mathcal{B}$. \\[2mm]
\indent Let $f = u+i v$ be a complex-valued function of $z =x +i y$ in a simply connected domain $\Omega$. If $f\in C^2(\Omega)$ (continuous first and second partial derivatives in 
 $\Omega$) and satisfies the Laplace equation $\Delta f = 4 f_{z\ol{z}}=0$ in $\Omega$, then $f$ is said to be harmonic in $\Omega$, where 
$f_z=(1/2)(\pa f/\pa x -i \pa f/\pa y)$ and $f_{\ol{z}}=(1/2)(\pa f/\pa x +i \pa f/\pa y)$. Note that every harmonic
 mapping $f$ has the canonical representation $f =h +\ol{g}$, where $h$ and $g$ are analytic in $\Omega$, known respectively as the analytic and co-analytic parts of $f$. 
 This representation is unique up to an additive constant (see \cite{D2004}). The inverse function theorem and a result of Lewy \cite{L1936} shows that a harmonic
 function $f$ is locally univalent in $\Omega$ if, and only if, $J_f(z)\not= 0$ in $\Omega$. A harmonic mapping f is locally univalent
 and sense-preserving in $\Omega$ if, and only if, $J_f(z)> 0$ in $\Omega$. Note that $|f_z|\not=0$ whenever $J_f>0$. 
\subsection{Log-harmonic mappings}
A log-harmonic mapping $f$ defined in $\mathbb{D}$ is a solution of the nonlinear elliptic partial differential equation
\bea\label{eq1}
\frac{\overline{f_{\overline{z}}(z)}}{\overline{f(z)}}=\omega(z)\frac{f_z(z)}{f(z)},~z\in\mathbb{D},\eea
where the second complex dilatation function $\omega:\mathbb{D}\to\mathbb{D}$ is analytic and the Jacobian of $f$ is given by $J_f=|f_z|^2-|f_{\overline{z}}|^2=|f_z|^2(1-|\omega|^2)>0$, which shows that non-constant log-harmonic mapping $f$ is always sense-preserving in $\mathbb{D}$. If $f$ is non-constant and vanishes only at $z = 0,$ then $f$ admits the following representation
\begin{align*}
f(z) = z^m|z|^{2\beta m}h(z)\overline{g(z)},
\end{align*}
where $m$ is non-negative integer, $\beta=\ol{\omega(0)}(1+\omega(0))/\left(1-|\omega(0)|^2\right)$, {\it i.e.,} $\textrm{Re}(\beta)>-1/2$, and $h$ and $g$ are analytic functions in $\mathbb{D}$ satisfying $g(0)=1$ and $h(0)\neq 0$. If $f$ is univalent 
log-harmonic mapping in $\mathbb{D}$, then either $0\not\in f(\mathbb{D})$ and $\log(f)$ is univalent and harmonic in $\mathbb{D}$ or, if $f(0)=0$, then 
$f=z |z|^{2\beta} h\ol{g}$, where $\text{Re}(\beta)>-1/2$, $0\not\in hg(\mathbb{D})$ and $F(t)=\log(f(e^t))$ is univalent and harmonic on the half-plane 
$\{t\in\mathbb{C}: \text{Re}(t)<0\}$. The class of such functions has been studied extensively in \cite{AB1988, AA2012, A1996}. Note that here ``$\log$'' denotes the principal branch of the logarithm. It is known that $F$ is closely 
related with the theory of nonparametric minimal surfaces over domains of the form $-\infty < u < u_0(v)$, $u_0(v+2\pi)=u_0(v)$ (see \cite{N1989}).\\[2mm]
\indent If $f$ is a non-vanishing log-harmonic mapping in $\mathbb{D}$, then $f$ can be expressed as
\bea\label{c1} f(z) = h(z)\overline{g(z)},\eea
where $h$ and $g$ are non-vanishing analytic functions in the unit disk $\mathbb{D}.$ Further, if the mapping $f$ given by \eqref{c1} is locally univalent and sense-preserving, then $h'g \neq 0$ in $\mathbb{D}$ and the second complex dilatation $\omega$ is given by $\omega=g'h/gh'$ is a Schwarz function in $\mathbb{D}$. Furthermore, it is evident that
\beas J_f(z)=|h'(z)g(z)|^2-|g'(z)h(z)|^2=|h'(z)g(z)|^2\left(1-\left|\frac{h(z)g'(z)}{h'(z)g(z)}\right|^2\right)>0,~z\in\mathbb{D}.\eeas
Several authors have studied and established fundamental results on log-harmonic mappings defined on the unit disk $\mathbb{D}$ (see \cite{BHPV2022, AA2012, 1LP2018}).
\subsection{Pre-Schwarzian and Schwarzian derivatives of analytic functions}
For a locally univalent analytic function $f$ defined in a simply connected domain $\Omega\subset \mathbb{C}$, the pre-Schwarzian derivative $P_f$ and the Schwarzian derivative $S_f$ of $f$ are, respectively, defined as follows:
\bea\label{c2} P_f(z)=\frac{f''(z)}{f'(z)}\quad\text{and}\quad S_f(z) = P_f'(z)-\frac{1}{2}P_f^2(z)=\frac{f'''(z)}{f''(z)}-\frac{3}{2}\left(\frac{f''(z)}{f'(z)}\right)^2.\eea
Moreover, the pre-Schwarzian and the Schwarzian norms of $f$ are, respectively, given by
\bea\label{r1} \Vert P_f\Vert= \sup_{z \in \mathbb{D}}(1-|z|^2)|P_f(z)|\quad\text{and}\quad \Vert S_f\Vert = \sup_{z \in \mathbb{D}}(1-|z|^2)^2|f_h(z)|.\eea
Some important global univalence criteria for a locally univalent analytic function have been obtained using the pre-Schwarzian and Schwarzian norms.
For a univalent analytic function $f$ in $\mathbb{D}$, it is well-known that $\Vert P_f\Vert\leq 6$ and the equality 
is attained for the Koebe function or its rotation. 
One of the most used univalence criterion for locally univalent analytic functions is the Becker's univalence criterion \cite{B1972}, 
which states that if $f$ a locally univalent analytic function and $\sup_{z\in\mathbb{D}}\left(1-|z|^2\right) \left|zP_f(z)\right|\leq1$, then $f$ is univalent in $\mathbb{D}$. In a subsequent study, 
Becker and Pommerenke \cite{BP1984} prove that the constant $1$ is sharp.
In 1976, Yamashita \cite{Y1976} 
proved that $\Vert P_f \Vert<\infty $ is finite if, and only if, $f$ is uniformly locally univalent in $\mathbb{D}$. Moreover, if $\Vert P_f\Vert<2$, then $f$ is bounded in $\mathbb{D}$ (see \cite{KS2002}).\\[2mm] 
\indent In terms of the Schwarzian derivative, it is well-known that for any univalent analytic function $f$ in $\mathbb{D}$, we have the sharp inequality $\Vert S_f\Vert \leq 6$ and
the equality is attained for the Koebe function or its rotation (see \cite{K1932}). 
In 1949, Z. Nehari \cite{N1949} established important criteria for global univalence, expressed in terms of the Schwarzian derivative, by virtue of the connection with linear differential equations.
For instance, if $f$ is locally univalent and analytic in $\mathbb{D}$ and satisfies $\Vert S_f\Vert \leq2$, then $f$ is univalent in $\mathbb{D}$. The bound $2$ is sharp \cite{H1949}.
\subsection{Pre-Schwarzian and Schwarzian norms of harmonic mappings}
\indent For a locally univalent harmonic mapping $f=h+\overline{g}$ in the unit disk $\mathbb{D}$, Hern\'andez and Mart\'in \cite{HM2015} have defined the pre-Schwarzian and Schwarzian derivatives, respectively, as follows:
\beas
P_f &=& \left(\log(J_f)\right)_z = \frac{h''}{h'}-\frac{\overline{\omega}\omega'}{1-|\omega|^2}\quad \text{and}\hfill\\[2mm]
S_f &=& \left(\log J_f\right)_{zz}-\frac{1}{2}\left(\log J_f\right)_z^2= S_h+\frac{\overline{\omega}}{1-|\omega|^2}\left(\frac{h''}{h'}\omega'-\omega''\right)-\frac{3}{2}\left(\frac{\omega'\overline{\omega}}{1-|\omega|^2}\right)^2,\eeas
where $S_h$ is the classical Schwarzian derivative of the analytic function $h$, $J_f$ is the Jacobian and $\omega=g'/h'$ is the second complex dilatation of $f$. These derivatives play a crucial role in understanding the behavior of harmonic functions, particularly in the context of conformal mappings and their geometric properties. 
This notion of pre-Schwarzian and Schwarzian derivatives of harmonic functions is a generalization of the classical pre-Schwarzian and Schwarzian derivatives of analytic functions. 
Note that when $f$ is analytic, we have $\omega=0.$ It is also easy to see that $S_f = (P_f)_z-(1/2)(P_f)^2.$ As in the case of analytic functions, for a sense-preserving 
locally univalent harmonic mapping $f = h+\bar{g}$ in the unit disk $\mathbb{D}$, the pre-Schwarzian norm $\Vert P_f\Vert$ and the Schwarzian norm $\Vert S_f\Vert$ are defined by (\ref{r1}). For a comprehensive study on the pre-Schwarzian and Schwarzian derivatives for harmonic mappings, we refer to \cite{HM2015, 2LP2018}. 
\subsection{Pre-Schwarzian and Schwarzian derivatives of a non-vanishing log-harmonic mapping}
\indent For a locally univalent non-vanishing log-harmonic mapping $f=h\ol{g}$ in a simply connected domain $\Omega$, Bravo {\it et al.} \cite{BHPV2022} have defined the pre-Schwarzian $P_f$ and Schwarzian derivatives $S_f$, respectively, as follows:
\beas
P_f&=&\left(\log(J_f)\right)_z =\frac{h''}{h'}+\frac{g'}{g}-\frac{\overline{\omega}\omega'}{1-|\omega|^2}=P_\phi-\frac{\overline{\omega}\omega'}{1-|\omega|^2}\text{and}\\[2mm]
S_f&=&\left(\log J_f\right)_{zz}-\frac{1}{2}\left(\log J_f\right)_z^2=S_\phi+\frac{\overline{\omega}}{1-|\omega|^2}\left(\frac{\phi''}{\phi'}\omega'-\omega''\right)-\frac{3}{2}
\left(\frac{\overline{\omega}\omega'}{1-|\omega|^2}\right)^2,\eeas
where $J_f$ is the Jacobian of $f$, $\phi'=h' g$ and $\omega=g'h/(h' g)$ is the second complex dilatation of the function $f$. If $f$ is a sense-preserving log-harmonic mapping of the form \eqref{c1} and $\psi$ is a locally univalent analytic function for which the composition $f\circ\psi$ is well defined, then the function $f\circ\psi$ is again a sense-preserving log-harmonic mapping, and $P_{f\circ\psi}=(P_f\circ\psi)\phi'+P_\psi$ and $S_{f\circ\psi}=(S_f\circ\psi)\phi'+S_\psi$.
For more information about the properties of the pre-Schwarzian and Schwarzian derivatives of a sense-preserving log-harmonic mapping, we refer to \cite{BHPV2022, AP2024, MPW2013}.
\section{Properties of the pre-Schwarzian norm of a non-vanishing log-harmonic mapping}
\noindent In the following result, we establish a necessary and sufficient condition for a non-vanishing log-harmonic mapping having a finite pre-Schwarzian norm.
\begin{thm}\label{Th0}
Let $f=h\overline{g}$ be a non-vanishing sense-preserving and locally univalent log-harmonic mapping with the dilatation $\omega=g'h/h'g$ defined in $\mathbb{D}$. Then either, $\Vert P_f \Vert =\Vert P_{hg}\Vert =\infty$ or, both $\Vert P_f \Vert $ and $\Vert P_{hg}\Vert $ are finite. If $\Vert P_f \Vert $ is finite, then
\beas
\left|\Vert P_f \Vert -\Vert P_{hg}\Vert \right|\leq 1\quad\text{and the constant $1$ is sharp.}\eeas
\end{thm}
\begin{proof}
Since $f=h\overline{g}$ is a non-vanishing sense-preserving and locally univalent log-harmonic mapping in $\mathbb{D}$, thus $h'g \neq 0$ in $\mathbb{D}$ and the second complex dilatation $\omega=g'h/gh'$ is a Schwarz function in $\mathbb{D}$. The pre-Schwarzian derivative of $f$ is given by
\bea\label{a2}
P_f(z)=\frac{h''(z)}{h'(z)}+\frac{g'(z)}{g(z)}-\frac{\overline{\omega(z)}\omega'(z)}{1-|\omega(z)|^2}.\eea
Let $f_1=h(z)g(z)$, $z\in\mathbb{D}$. Since $h$ and $g$ are non-vanishing analytic functions, therefore $f_1$ is a non-vanishing analytic function with 
\bea\label{b2} f'_1(z)=h'(z)g(z)+h(z)g'(z)=h'(z)g\left(1+\;\omega(z) \right),\eea
which shows that $f_1$ is locally univalent function in $\mathbb{D}$. Taking logarithmic derivative on both sides of (\ref{b2}) with respect to $z$, we obtain
\beas P_{f_1}(z)=\frac{f''_1(z)}{f'_1(z)}=\frac{h''(z)}{h'(z)}+\frac{g'(z)}{g(z)}+\frac{\omega'(z) }{1+\omega(z) }\eeas
Therefore, the difference between the pre-Schwarzian derivatives of $f_1$ and $f$ is given by
\bea\label{b4}
P_{hg}(z)-P_f(z)=\frac{\omega'(z) }{1+\omega(z) }+\frac{\overline{\omega(z)}\omega'(z)}{1-|\omega(z)|^2}=\frac{1+\ol{\omega(z)} }{1+\;\omega(z) }\cdot\frac{\omega'(z)}{1-|\omega(z)|^2}. \eea
In view of the Schwarz-Pick lemma, we have
\beas
(1-|z|^2)\left||P_{hg}(z)|-|P_f|\right|&\leq & (1-|z|^2)\left|P_{hg}(z)-P_f(z)\right| \\
&=&  (1-|z|^2)\left|\frac{1+\ol{\omega(z)}}{1+\omega(z) }\cdot\frac{\omega'(z)}{1-|\omega(z)|^2}\right|\\
&\leq &1,\eeas
which shows that $\Vert P_f\Vert$ is finite if, and only if, $\Vert P_{hg}\Vert$ is finite. Moreover, if $\Vert P_f\Vert<\infty$, then
\beas
\left|\Vert P_f\Vert-\Vert P_{hg}\Vert\right|\leq \sup_{z\in\mathbb{D}} (1-|z|^2)\left||P_f(z)|-|P_{hg}(z)|\right|\leq 1.\eeas
\indent To show that the constant $1$ is sharp, we consider the log-harmonic mapping $f(z)=h(z)\overline{g(z)}$ with $h(z)=\exp(z/(1-z))$ and $g(z)=\exp(-z/(1-z))/(1-z)$. It is evident that the dilatation $\omega(z)=-z$ and $h(z)g(z)=1/(1-z)$. 
Therefore, the pre-Schwarzian derivatives of $f$ and $hg$, respectively, are
\beas P_f(z)=\frac{3}{(1-z)}-\frac{\overline{z}}{1-|z|^2}\quad \text{and}\quad P_{hg}(z)=\frac{2}{1-z}.\eeas
Thus, the pre-Schwarzian norm of $f$ is
\beas
\Vert P_f\Vert=\sup_{z\in\mathbb{D}}(1-|z|^2)\left|P_f\right|=\sup_{z\in\mathbb{D}}(1-|z|^2)\left|\frac{3}{1-z}-\frac{\overline{z}}{1-|z|^2}\right|.
\eeas
On the positive real axis, we have
\beas \sup_{0\leq r<1} (1-r^2)\left|\frac{3}{1-r}-\frac{r}{1-r^2}\right|=5.\eeas
Similarly, we have $\Vert P_{hg}\Vert=4$. 
Hence, we have $\left|\Vert P_f \Vert -\Vert P_{hg}\Vert \right|=|5-4|=1$, which shows that the constant $1$ is sharp.
This completes the proof.
\end{proof}
\noindent Motivated by the results of Liu and Ponnusamy \cite{2LP2018}, we establish a relationship between the pre-Schwarzian norm of a non-vanishing log-harmonic mapping $f = h\overline{g}$ and that of the analytic function $hg^{\epsilon}$, where $\epsilon \in \overline{\mathbb{D}}$ and we choose the principal branch of the logarithm.
\begin{thm}\label{Th1}
Let $f=h\overline{g}$ be a non-vanishing sense-preserving and locally univalent log-harmonic mapping with the dilatation $\omega=g'h/h'g$ defined in $\mathbb{D}$. If $\log{g(z)}$ is an analytic Bloch function in $\mathbb{D}$, then either, $\Vert P_f \Vert =\Vert P_{hg^{\epsilon}}\Vert =\infty$ or, both $\Vert P_f \Vert $ and $\Vert P_{hg^{\epsilon}}\Vert $ are finite for each $\epsilon\in\ol{\mathbb{D}}$. If $\Vert P_f \Vert $ is finite, then
\beas
\left|\Vert P_f \Vert -\Vert P_{hg^{\epsilon}}\Vert \right|\leq 1+|1-\epsilon|\;\beta_{\log{g}}\leq 1+2\;\beta_{\log{g}},
\eeas
for each $\epsilon\in\ol{\mathbb{D}}$. 
In particular,
\beas
\left|\Vert P_f \Vert -\Vert P_{h}\Vert \right|\leq 1+\beta_{\log{g}}.\eeas
\end{thm}
\begin{proof}
As $f=h\overline{g}$ is a non-vanishing sense-preserving and locally univalent log-harmonic mapping in $\mathbb{D}$, it follows that $h'g \neq 0$ in $\mathbb{D}$ and the second complex dilatation $\omega=g'h/gh'$ is a Schwarz function in $\mathbb{D}$. The pre-Schwarzian derivative of $f$ is given in (\ref{a2}).
Let $f_\epsilon(z)=h(z)g^{\epsilon}(z)$ defined in $\mathbb{D}$, where $\epsilon\in\ol{\mathbb{D}}$. Since $h$ and $g$ are non-vanishing analytic functions, therefore $f_\epsilon$ is a non-vanishing analytic function with 
\bea\label{a1} f'_\epsilon(z)=h'(z)g^{\epsilon}(z)+\epsilon h(z)g^{\epsilon-1}(z)g'(z)=h'(z)g^\epsilon(z)\left(1+\epsilon\;\omega(z) \right),\eea
which shows that $f_\epsilon$ is locally univalent function in $\mathbb{D}$. Taking logarithmic derivative on both sides of (\ref{a1}) with respect to $z$, we obtain
\beas \frac{f''_\epsilon(z)}{f'_\epsilon(z)}=\frac{h''(z)}{h'(z)}+\epsilon \frac{g'(z)}{g(z)}+\frac{\epsilon\;\omega'(z) }{1+\epsilon\;\omega(z) }\eeas
 Therefore, the pre-Schwarzian derivative $f_{\epsilon}$ is given by
\bea\label{a3}
P_{f_\epsilon}(z)=\frac{f''_\epsilon(z)}{f'_\epsilon(z)}=\frac{h''(z)}{h'(z)}+ \frac{g'(z)}{g(z)}+(\epsilon-1) \frac{g'(z)}{g(z)}+\frac{\epsilon\;\omega'(z) }{1+\epsilon\;\omega(z) }
\eea
From (\ref{a2}) and (\ref{a3}), we have
\bea\label{a4}
P_{hg^{\epsilon}}(z)-P_f(z)&=&(\epsilon-1) \frac{g'(z)}{g(z)}+\frac{\epsilon\;\omega'(z) }{1+\epsilon\;\omega(z) }+\frac{\overline{\omega(z)}\omega'(z)}{1-|\omega(z)|^2}\nonumber\\[2mm]
&=&(\epsilon-1) \frac{g'(z)}{g(z)}+\frac{\epsilon+\ol{\omega(z)} }{1+\epsilon\;\omega(z) }\cdot\frac{\omega'(z)}{1-|\omega(z)|^2}. \eea
In view of the Schwarz-Pick lemma, we have
\beas
(1-|z|^2)\left||P_{hg^{\epsilon}}(z)|-|P_f|\right|&\leq & (1-|z|^2)\left|P_{hg^{\epsilon}}(z)-P_f(z)\right| \\
&=&  (1-|z|^2)\left|(\epsilon-1) \frac{g'(z)}{g(z)}+\frac{\epsilon+\ol{\omega(z)}}{1+\epsilon\;\omega(z) }\cdot\frac{\omega'(z)}{1-|\omega(z)|^2}\right|\\
&\leq &|1-\epsilon|\;\sup_{z\in\mathbb{D}}(1-|z|^2)\left| \frac{g'(z)}{g(z)}\right|+\sup_{z\in\mathbb{D}}\left|\frac{\ol{\epsilon}+\omega(z) }{1+\epsilon\;\omega(z) }\right|\\[2mm]
&\leq&1+|1-\epsilon|\;\sup_{z\in\mathbb{D}}(1-|z|^2)\left| \frac{g'(z)}{g(z)}\right|.\eeas
If $\log{g(z)}$ is an analytic Bloch function in $\mathbb{D}$, then
\beas \beta_{\log{g}}:=\sup_{z\in\mathbb{D}}(1-|z|^2)\left|\frac{g'(z)}{g(z)}\right|<\infty.\eeas
Therefore, we have
\beas (1-|z|^2)\left||P_{h(z)g^{\epsilon}(z)}|-|P_f|\right|\leq  1+|1-\epsilon|\beta_{\log{g}}\leq 1+2\beta_{\log{g}}<\infty,\eeas
which shows that $\Vert P_f\Vert$ is finite if, and only if, $\Vert P_{h(z)g^{\epsilon}}\Vert$ is finite. Moreover, if $\Vert P_f\Vert<\infty$, then
\beas
\left|\Vert P_f\Vert-\Vert P_{hg^{\epsilon}}\Vert\right|\leq \sup_{z\in\mathbb{D}} (1-|z|^2)\Big||P_f(z)|-|P_{hg^{\epsilon}}(z)|\Big|\leq 1+ |1-\epsilon|\beta_{\log{g}}.
\eeas
This completes the proof.
\end{proof}
This relationship in \textrm{Theorem \ref{Th0}} not only enhances our understanding of the behavior of log-harmonic mappings but also provides insights into their geometric properties.
Now, if we consider a log-harmonic mapping $f=h\overline{g}$ in $\mathbb{D}$ such that $h$ is an analytic and locally univalent function, and $g$ is a non-vanishing analytic function in $\mathbb{D}$, then $h'g \neq 0$ in $\mathbb{D}$. This shows that $f$ is locally univalent in $\mathbb{D}$.
\begin{thm}\label{Th11}
Let $f=h\overline{g}$ be a sense-preserving log-harmonic mapping such that $h$ is analytic and locally univalent, and $g$ is non-vanishing analytic in $\mathbb{D}$ with the dilatation $\omega=g'h/h'g$ defined in $\mathbb{D}$. If $\log{g(z)}$ is an analytic Bloch function in $\mathbb{D}$, then either, $\Vert P_f \Vert =\Vert P_{hg^{\epsilon}}\Vert =\infty$ or, both $\Vert P_f \Vert $ and $\Vert P_{hg^{\epsilon}}\Vert $ are finite for each $\epsilon\in\ol{\mathbb{D}}$. If $\Vert P_f \Vert $ is finite, then
\beas
\left|\Vert P_f \Vert -\Vert P_{hg^{\epsilon}}\Vert \right|\leq 1+2\;\beta_{\log{g}},
\eeas
for each $\epsilon\in\ol{\mathbb{D}}$. 
In particular,
\beas
\left|\Vert P_f \Vert -\Vert P_{h}\Vert \right|\leq 1+\beta_{\log{g}}.\eeas
Both the constants $ 1+2\;\beta_{\log{g}}$ and $1+\beta_{\log{g}}$ are sharp.
\end{thm}
\begin{proof} Using analogous reasoning to that used to prove \textrm{Theorem} \ref{Th1}, we arrive at the desired conclusion.\\[2mm]
\indent To show that the constant $1+2\beta_{\log g}$ is sharp, we consider the log-harmonic mapping $f(z)=h(z)\overline{g(z)}$ where $h(z)=z/(1-z)$, $g(z)=1/(1-z)$. It is 
evident that $h(z)$ is an analytic and locally univalent function while $g(z)$ is non-vanishing analytic in $\mathbb{D}$. As $g'(z)h(z)/(g(z)h'(z))=\omega(z)$, thus, we have the 
dilatation $\omega(z)=z$. Let $\epsilon=-1$. The pre-Schwarzian derivatives of $f$ and $hg^{-1}$ are, respectively, given by 
\beas
P_f(z)=\frac{3}{1-z}-\frac{\overline{z}}{1-|z|^2}\quad\text{and}\quad P_{hg^{-1}}(z)=0.\eeas
It is evident that
\beas
\beta_{\log{g}}=\sup_{z\in\mathbb{D}}(1-|z|^2)\left|\frac{g'(z)}{g(z)}\right|=\sup_{z\in\mathbb{D}}(1-|z|^2)\left|\frac{1}{1-z}\right|=2,\eeas
which shows that $\log{g}$ is analytic Bloch function. Thus, the pre-Schwarzian norms of $f$ and $hg^{-1}$ are, respectively, given by 
\beas
\Vert P_f\Vert=5\quad\text{and}\quad \Vert P_{hg^{-1}}\Vert=0.\eeas
Therefore, $\left|\Vert P_f\Vert-\Vert P_{hg^{-1}}\Vert\right|=5=1+2\cdot2$.\\[2mm]
\indent To show that the constant $1+\beta_{\log g}$ is sharp, we consider the log-harmonic mapping $f(z)=h(z)\overline{g(z)}$ with $h(z)=1/(1-z)$ and dilatation $\omega(z)=(a-z)/(1-az)$, where $a\in(0,1)$. As $\omega=g'h/h'g$, we have 
\beas \frac{g'(z)}{g(z)}=\frac{a-z}{(1-az)(1-z)}.\eeas
It is evident that
\beas
\beta_{\log{g}}=\sup_{z\in\mathbb{D}}(1-|z|^2)\left|\frac{g'(z)}{g(z)}\right|=\sup_{z\in\mathbb{D}}(1-|z|^2)\left|\frac{a-z}{(1-az)(1-z)}\right|.\eeas
On the positive real axis, we have 
\beas \sup_{0\leq r<1}\frac{(a-r)(1+r)}{(1-ar)}=2.\eeas
Therefore, $\log{g}$ is analytic Bloch function and $\beta_{\log{g}}=2$. It is evident that 
\beas 1-|\omega(z)|^2&=&\frac{(1-a^2)(1-|z|^2)}{(1-az)(1-a\ol{z})}.\eeas
Therefore, the pre-Schwarzian derivative of $f$ is 
\beas P_f(z)&=&\frac{h''(z)}{h'(z)}+\frac{g'(z)}{g(z)}-\frac{\overline{\omega(z)}\omega'(z)}{1-|\omega(z)|^2}\\[2mm]
&=&\frac{2}{1-z}+\frac{a-z}{(1-az)(1-z)}+\frac{(a-\ol{z})}{(1-az)(1-|z|^2)}\eeas
Similarly, a direct computation shows that $P_h(z)=h''(z)/h'(z)=2/(1-z)$ and the pre-Schwarzian norm of $h$ is $\Vert P_h\Vert=4$. 
The pre-Schwarzian norm of $f$ is 
\beas\Vert P_f\Vert=\sup_{z\in\mathbb{D}}\;(1-|z|^2)\left|\frac{2}{1-z}+\frac{a-z}{(1-az)(1-z)}+\frac{a-\ol{z}}{(1-az)(1-|z|^2)}\right|.\eeas
On the positive real axis, we have 
\beas \sup_{0\leq r<1}\left|2(1+r)+\frac{(a-r)(1+r)}{1-a r}+\frac{a-r}{1-a r}\right|=\sup_{0\leq r<1}G(a, r),\eeas
where $G(a,r)=2(1+r)+(a-r)(2+r)/(1-a r)$. Differentiate $G(a, r)$ partially with respect to $r$, we obtain
\beas \frac{\partial }{\partial r}G(a, r)=\frac{\left(2 a+1\right)\left(a r^2-2 r+a\right)}{(1-a r)^2}.\eeas
The roots of the equation $a r^2-2 r+a=0$ are 
\beas r_1(a)=\frac{1-\sqrt{1-a^2}}{a}\quad\text{and}\quad r_2(a)=\frac{1+\sqrt{1-a^2}}{a}.\eeas
It is evident that $r_2(a)>1$ and $r_1(a)\in(0, 1)$ for $a\in(0, 1)$. Hence, we have  
\beas\Vert P_f\Vert=G(a, r_1(a))=\frac{4 a^3+\left(\sqrt{1-a^2}+2\right) a^2+(4a+2) \left(\sqrt{1-a^2}-1\right)}{a^2 \sqrt{1-a^2}},\eeas
which tends to $7$ as $a\to1^-$. Therefore, we see that 
\beas \left|\Vert P_f \Vert -\Vert P_{h}\Vert \right|=3= 1+\beta_{\log{g}}.\eeas
This completes the proof.
\end{proof}
In the following result, our objective is to ascertain a condition under which the analytic function $hg^\epsilon$ is univalent in $\mathbb{D}$, where $\epsilon\in\ol{\mathbb{D}}$.
\begin{thm}\label{Th4}
Let $f=h\overline{g}$ be a non-vanishing sense-preserving and locally univalent log-harmonic mapping with the dilatation $\omega=g'h/h'g$ defined in $\mathbb{D}$. If 
\bea\label{a5} |z P_{f}|+|1-\epsilon|\left|\frac{z g'(z)}{g(z)}\right|+\frac{|z \omega'(z)|}{1-|\omega(z)|^2}\leq \frac{1}{(1-|z|^2)}\eea
for $\epsilon\in\ol{\mathbb{D}}$, then $f_\epsilon=hg^\epsilon$ is univalent in $\mathbb{D}$.
\end{thm}
\begin{proof}
Using similar argument as in the proof of \textrm{Theorem} \ref{Th1}, we obtain (\ref{a4}). Thus, we have
\bea\label{a6}
\left|zP_{hg^{\epsilon}}(z)\right|\leq\left|zP_f(z)\right|+|1-\epsilon|\left|\frac{zg'(z)}{g(z)}\right|+\left|\frac{\epsilon+\ol{\omega(z)} }{1+\epsilon\;\omega(z) }\cdot\frac{z\omega'(z)}{1-|\omega(z)|^2}\right|.\eea
Using (\ref{a5}) and (\ref{a6}), we have
\beas \sup_{z\in\mathbb{D}}(1-|z|^2)\left|z P_{hg^\epsilon}\right|&\leq&   \sup_{z\in\mathbb{D}}(1-|z|^2)|z P_{f}|+|1-\epsilon|\sup_{z\in\mathbb{D}}(1-|z|^2)\left|\frac{z g'(z)}{g(z)}\right|\\[2mm]
&&+\sup_{z\in\mathbb{D}}\left|\frac{\ol{\epsilon}+\omega(z) }{1+\epsilon\;\omega(z) }\right|\cdot\sup_{z\in\mathbb{D}}\frac{|z \omega'(z)|(1-|z|^2)}{1-|\omega(z)|^2}\\[2mm]
&\leq& \sup_{z\in\mathbb{D}}(1-|z|^2)|z P_{f}|+|1-\epsilon|\sup_{z\in\mathbb{D}}(1-|z|^2)\left|\frac{z g'(z)}{g(z)}\right|\\[2mm]
&&+\sup_{z\in\mathbb{D}}\frac{|z \omega'(z)|(1-|z|^2)}{1-|\omega(z)|^2}\\
&\leq&1.\eeas
In view of the Becker's univalence criterion, $hg^\epsilon$ is univalent in $\mathbb{D}$. This completes the proof.
\end{proof} 
The following result is an immediate consequence of \textrm{Theorem} \ref{Th4}.
\begin{cor}\label{Th5}
Let $f=h\overline{g}$ be a non-vanishing sense-preserving and locally univalent log-harmonic mapping with the dilatation $\omega=g'h/h'g$ defined in $\mathbb{D}$. If 
\beas |z P_{f}|+\frac{|z \omega'(z)|}{1-|\omega(z)|^2}\leq \frac{1}{(1-|z|^2)},\eeas
then $hg$ is univalent in $\mathbb{D}$ and $\Vert P_f\Vert\leq 7$.
\end{cor}
\begin{proof}  Using analogous reasoning to that used to prove \textrm{Theorem} \ref{Th4}, we obtain $hg$ is univalent in $\mathbb{D}$.
Therefore, we have $\Vert P_{hg}\Vert \leq6$ and it follows from \textrm{Theorem \ref{Th0}}, we have $\Vert P_f\Vert \leq 7$. This completes the proof.
\end{proof}
Corollary \ref{Th5} provides a bound of the pre-Schwarzian norm $\Vert P_f\Vert$ of a non-vanishing sense-preserving and locally univalent log-harmonic mapping $f$ in $\mathbb{D}$.
For the purposes of this paper, the following definition is required.
\begin{defi} A univalent function $f\in C^1(\mathbb{D})$ (continuous first order partial derivatives in $\mathbb{D}$) with $f(0)=0$ is called starlike with respect to the origin if
\beas\frac{\partial }{\partial \theta}\left(f(r e^{i\theta })\right)=\text{Re}\left(\frac{z f_{z}-\ol{z} f_{\ol{z}}}{f}\right)>0\quad\text{for all $z=r e^{i\theta}\in\mathbb{D}\setminus\{0\}$.}\eeas
\end{defi}
\noindent We denote by $S^*_{Lh}$ (res., $S^*$) the set of all starlike log-harmonic (res., analytic) mappings with respect to the origin. This definition generalizes the standard notion of starlikeness for analytic functions to the log-harmonic setting and is equivalent to the condition that the curve $f(r e^{i\theta })$ is starlike for each $r\in(0, 1)$. For further details, we refer to \cite{AH1987}.\\[2mm]
\noindent Abdulhadi and Hengartner \cite{AH1987} have established the following result for mappings in $S^*_{Lh}$.
\begin{thmA}\cite[Theorem 2.1]{AH1987}
\begin{enumerate}
 \item[(a)] If $f(z)=z|z|^{2\beta}h(z)\ol{g(z)}\in S^*_{Lh}$, then $\phi(z)=zh(z)/g(z)\in S^*$.  
 \item[(b)] For any given $\phi\in S^*$ and $\omega\in \mathcal{B}(\mathbb{D})$, there are $h$ and $g$ in $\mathcal{H}(\mathbb{D})$ uniquely determined such that 
 \item[(i)] $0\not\in hg(\mathbb{D})$ with $h(0)=g(0)=1$,
 \item[(ii)] $\phi(z)=zh(z)/g(z)$,
 \item[(iii)]$f(z)=z|z|^{2\beta} h(z)\ol{g(z)}$ is a solution of (\ref{eq1}) in $S^*_{Lh}$, where $\beta=\ol{a(0)}(1+a(0))/(1-|a(0)|^2)$.
 \end{enumerate}
 \end{thmA}
Now, we provide an example of univalent log-harmonic mapping in $\mathbb{D}$ that vanishes at the origin.
\begin{ex}\label{Ex1} Let $h(z)=1/(1-z)$, $g(z)=1-z$, $\omega(z)=(2-3z)/(3-2z)$ and $\phi(z)=z/(1-z)^2$. It is evident that $\phi$ is the well-known Koebe function which is univalent and analytic in $\mathbb{D}$. Now,
\beas &&\text{Re}\left(\frac{z h'(z)}{h(z)}\right)=\text{Re}\left(\frac{z}{1-z}\right)>-\frac{1}{2},\\[2mm]
&&\text{Re}\left(\frac{z \phi'(z)}{\phi(z)}\right)=\text{Re}\left(\frac{1+z}{1-z}\right)=1+2\;\text{Re}\left(\frac{z}{1-z}\right)>0,\eeas
which shows that $\phi\in S^*$. Furthermore, it is easy to show that $\omega(z)$ is analytic in $\mathbb{D}$ with $|\omega(z)|<1$ for $z\in\mathbb{D}$, as illustrated in {\bf Figure} \ref{Fig1}. In view of \textrm{Theorem A}, we have $\beta =2$ and 
\bea\label{eq2} f(z)=\frac{z^3}{(1-z)}\ol{z^2(1-z)},\eea
which is a starlike log-harmonic mapping in $\mathbb{D}$ and hence, $f(z)$ is univalent in $\mathbb{D}$. Note that, the log-harmonic mapping $f(z)$ given by (\ref{eq2}) is univalent but has vanishes multiply at the origin. 
\begin{figure}[H]
\centering
\includegraphics[scale=0.8]{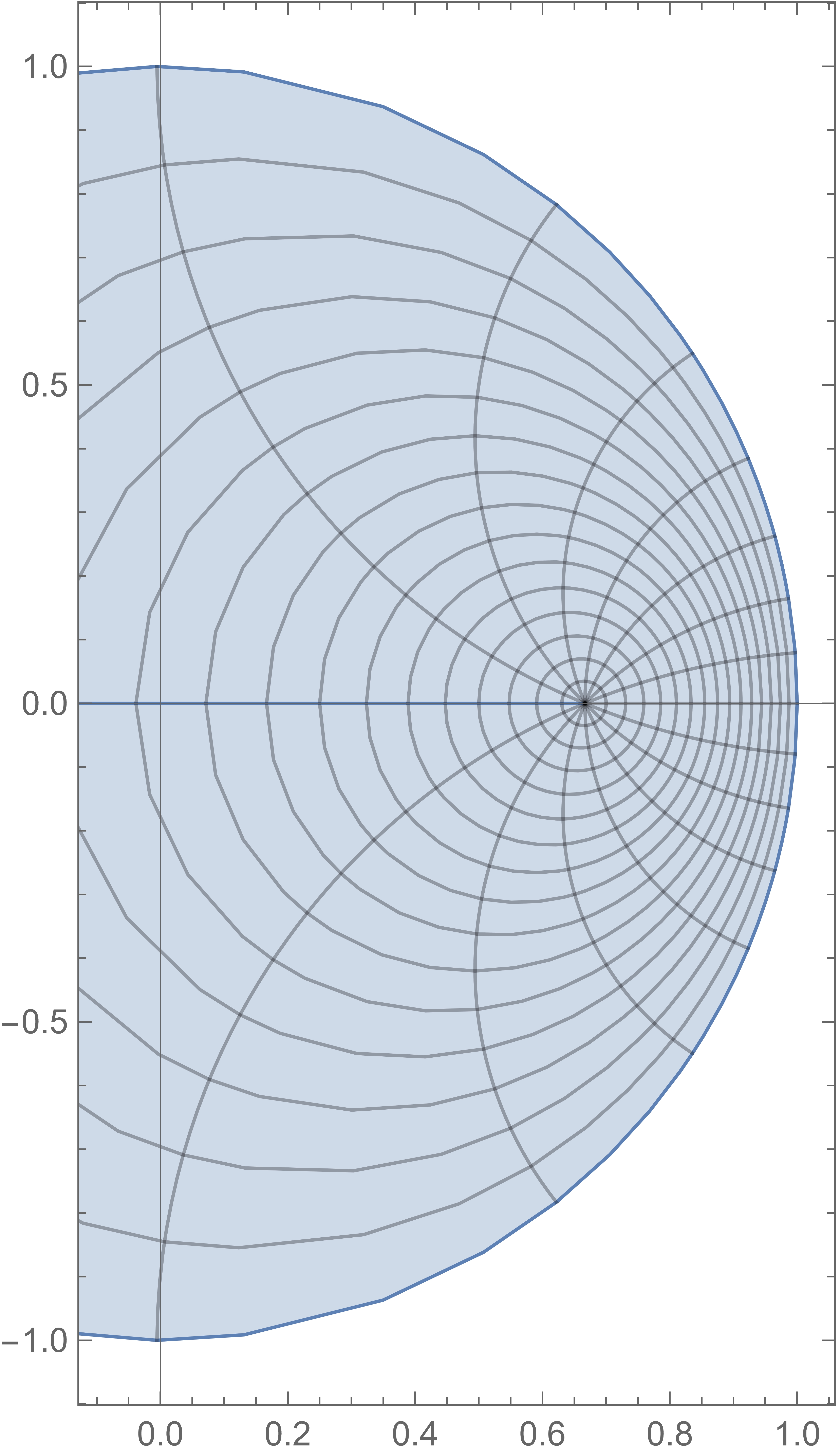}
\caption{The image of $\mathbb{D}$ under the mapping $\omega(z)=(2-3z)/(3-2z)$}
\label{Fig1}
\end{figure}
\end{ex}
\section{Derivation of the formula of pre-Schwarzian and Schwarzian derivatives of any log-harmonic mappings}
\noindent{\bf Motivation:} It is evident that $f(z)=z/|1-z|^2$ is a log-harmonic mapping in $\mathbb{D}$ that vanishes at the origin. Now, we can express $f(z)$ as $f(z)=zh(z)\ol{g(z)}$, where $h(z)=g(z)=1/(1-z)$ are analytic in $\mathbb{D}$ with $g(0)=1=h(0)$ and the second complex dilatation is 
\beas \omega(z)=\frac{z g'(z)/g(z)}{1+z h'(z)/h(z)}=z,\eeas
which is a Schwarz function in $\mathbb{D}$. The Jacobian of the log-harmonic mapping $f(z)=z/|1-z|^2$ in $\mathbb{D}$ is defined by 
\beas J_f=|f_z|^2-|f_{\ol{z}}|^2=\left|(zh'+h)g\right|^2\left(1-|\omega|^2\right).\eeas
Now, it is evident that $(zh'+h)g=1/(1-z)^3\not=0$ in $\mathbb{D}$ and it follows that $ J_f>0$ in $\mathbb{D}$. Thus, $f(z)=z/|1-z|^2$ is a locally univalent and 
sense-preserving log-harmonic mapping in $\mathbb{D}$ vanishes at the origin. 
Some other examples are univalent log-harmonic Koebe function 
\[f(z)=\frac{z|1-z|^2}{(1-z)^2}\exp\left(\frac{2z}{1-z}\right)\ol{\exp\left(\frac{2z}{1-z}\right)}\] and univalent log-harmonic right half-plane mapping 
\[f(z)=\frac{z}{1-z}\exp\left(\frac{z}{1-z}\right)\ol{\exp\left(\frac{z}{1-z}\right)}.\] 
For more other examples, we refer to \cite{LP2022}. Note that Bravo {\it et al.} \cite{BHPV2022} have defined the pre-Schwarzian and Schwarzian derivatives 
only for a locally univalent non-vanishing log-harmonic mapping $f=h\ol{g}$ in a simply connected domain. 
The log-harmonic mapping (\ref{eq2}) in \textrm{Example \ref{Ex1}} and other locally univalent or univalent log-harmonic mappings that vanish at the origin motivate us to find the 
pre-Schwarzian and Schwarzian derivatives of any locally univalent log-harmonic mappings in $\mathbb{D}$.\\[2mm]
\indent Let $f(z)=z^m|z|^{2\beta m} h(z)\ol{g(z)}$ be a locally univalent and sense-preserving log-harmonic mapping in $\mathbb{D}$, where 
$m$ is a non-negative integer, $\text{Re}(\beta)>-1/2$, and $h$ and $g$ are analytic in $\mathbb{D}$ such that $g(0)=1$ and $h(0)\not=0$. As $f$ is locally univalent and 
sense-preserving, thus, we have $|f_z|\not=0$. Note that $z^a=\exp(a\log(z))$, where the branch of the logarithm is determined by $log(1)=0$, {\it i.e.}, $1^a=1$. This ensure that the function is
 single-valued and analytic. Thus, the function $z^a$ is differentiable.
The Jacobian of the locally univalent and sense-preserving log-harmonic mapping $f=z^m|z|^{2\beta m} h(z)\ol{g(z)}$ in $\mathbb{D}$ is defined by 
\beas J_f&=&\left|\left(z^{(\beta+1)m}h'+(\beta+1)m\; z^{(\beta+1)m-1}h\right)\ol{z^{\beta m} g}\right|^2\\[2mm]
&&-\left|z^{(\beta+1)m}h \ol{\left(z^{\beta m}g'+\beta m z^{\beta m-1}g\right)}\right|^2\\[2mm]
&=&\left|(zh'+(\beta+1)m h)z^{(2\beta+1)m-1}g\right|^2-\left|z^{(2\beta+1)m-1}h\left(z g'+\beta m g\right)\right|^2\\[2mm]
&=&\left|(zh'+(\beta+1)m h)z^{(2\beta+1)m-1}g\right|^2\left(1-|\omega|^2\right),\eeas
where its second complex dilatation is given by 
\beas \omega(z)=\frac{z^{(2\beta+1)m-1}h\left(z g'+\beta m g\right)}{(zh'+(\beta+1)m h)z^{(2\beta+1)m-1} g}=\frac{z g'(z)/g(z)+\beta m}{(\beta+1)m+z h'(z)/h(z)}.\eeas
When $f$ is analytic, {\it i.e.,} $g\equiv 1$, $\beta=0$ and $z^m h=f$, then $J_f=|f'|^2$. In this case, the classical formulas of pre-Schwarzian and Schwarzian derivatives (\ref{c2}) of $f$ can be written as
\beas P_f=\frac{\partial}{\partial z}\log(J_f)\quad\text{and}\quad S_f=\frac{\partial^2}{\partial z^2}\log(J_f)-\frac{1}{2}\left(\frac{\partial}{\partial z}\log(J_f)\right)^2.\eeas
 This provides us the pre-Schwarzian derivative of $f=z^m |z|^{2\beta m} h(z)\ol{g(z)}$ as
 \beas P_f&=&\frac{\partial}{\partial z}\log\left(|(zh'+(\beta+1)m h)z^{(2\beta+1)m-1}g|^2\left(1-|\omega|^2\right)\right)\\[2mm]
&=&\frac{z^{(2\beta+1)m-1}g'+((2\beta+1)m-1) z^{(2\beta+1)m-2}g}{z^{(2\beta+1)m-1}g}+\frac{zh''+((\beta+1)m+1) h'}{zh'+(\beta+1)m h}\\[2mm]
&& +\frac{1}{1-|\omega|^2}\left(-\omega'\ol{\omega}\right).\eeas
It is evident that $P_f$ can also have the following form:
 \bea\label{c3} P_f&=&\frac{G'}{G}+\frac{H'}{H}-\frac{\ol{\omega}\omega'}{1-|\omega|^2},~G=z^{(2\beta+1)m-1}g,~ H=zh'+(\beta+1)m h.\eea
Now, the Schwarzian derivative of $f=z^m |z|^{2\beta m} h(z)\ol{g(z)}$ is given by
 \beas S_f&=&\frac{\partial^2}{\partial z^2}\log(J_f)-(1/2)\left(\frac{\partial}{\partial z}\log(J_f)\right)^2\\
 &=&\left(\frac{G'}{G}\right)'+\left(\frac{H'}{H}\right)'-\ol{\omega}\left(\frac{\omega'}{1-|\omega|^2}\right)'-\frac{1}{2}\left(\frac{G'}{G}+\frac{H'}{H}-\frac{\ol{\omega}\omega'}{1-|\omega|^2}\right)^2\\[2mm]
 &=&\frac{GG''-(G')^2}{G^2}+\frac{HH''-(H')^2}{H^2}-\ol{\omega}\frac{\omega''\left(1-|\omega|^2\right)+\ol{\omega}\left(\omega'\right)^2}{\left(1-|\omega|^2\right)^2}-\frac{1}{2}\left(\frac{G'}{G}\right)^2\\[2mm]
 &&-\frac{1}{2}\left(\frac{H'}{H}\right)^2-\frac{1}{2}\left(\frac{\ol{\omega}\omega'}{1-|\omega|^2}\right)^2-\frac{G'H'}{GH}+\frac{G'}{G}\frac{\ol{\omega}\omega'}{1-|\omega|^2}+\frac{H'}{H}\frac{\ol{\omega}\omega'}{1-|\omega|^2}\eeas
\beas &=&\left(\frac{G''}{G}-\frac{3}{2}\left(\frac{G'}{G}\right)^2\right)+\left(\frac{H''}{H}-\frac{3}{2}\left(\frac{H'}{H}\right)^2\right)-\frac{G'H'}{GH}-\frac{3}{2}\left(\frac{\ol{\omega}\omega'}{1-|\omega|^2}\right)^2\\[2mm]
 &&+\frac{\ol{\omega}}{1-|\omega|^2}\left(\left(\frac{H'}{H}+\frac{G'}{G}\right)\omega'-\omega''\right).\eeas
As $|f_z|\not=0$ in $\mathbb{D}$, it follows that $HG=(zh'+(\beta+1)mh)z^{(2\beta+1)m-1}g\not=0$ in $\mathbb{D}$. Thus, the analytic function defined by 
 \beas \phi(z)=\int_{z_0}^z (t h'(t)+(\beta+1)m h(t))t^{(2\beta+1)m-1}g(t) dt, \quad z_0\in\mathbb{D},\eeas
is locally univalent in $\mathbb{D}$ such that $\phi'(z)=(zh'+(\beta+1)m h)z^{(2\beta+1)m-1}g=HG$.
Therefore, the pre-Schwarzian and the Schwarzian derivatives of $\phi$ are, respectively, given below:
\beas P_\phi&=&\frac{(HG)'}{HG}=\frac{H'}{H}+\frac{G'}{G}\quad\text{and}\\[2mm]
S_\phi&=&\frac{HH''-(H')^2}{H^2}+\frac{GG''-(G')^2}{G^2}-\frac{1}{2}\left(\left(\frac{H'}{H}\right)^2+\left(\frac{G'}{G}\right)^2+2\frac{HG'}{HG}\right)\\[2mm]
&=&\left(\frac{G''}{G}-\frac{3}{2}\left(\frac{G'}{G}\right)^2\right)+\left(\frac{H''}{H}-\frac{3}{2}\left(\frac{H'}{H}\right)^2\right)-\frac{HG'}{HG}.\eeas
Thus, the pre-Schwarzian and Schwarzian derivatives of $f=z^m |z|^{2\beta m} h(z)\ol{g(z)}$ can be expressed as follows:
\be\label{c4}P_f=P_\phi-\frac{\ol{\omega}\omega'}{1-|\omega|^2}\quad\text{and}~ S_f=S_\phi-\frac{3}{2}\left(\frac{\ol{\omega}\omega'}{1-|\omega|^2}\right)^2+\frac{\ol{\omega}}{1-|\omega|^2}\left(\omega' P_\phi-\omega''\right).\ee
\begin{rem} Note that if we put $m=0$ in (\ref{c4}), we easily obtain the pre-Schwarzian and Schwarzian derivative of the non-vanishing log-harmonic mapping, which has been established in \cite{BHPV2022}.\end{rem}
Whenever we use the representation $f(z)=z^m|z|^{2\beta m} h(z)\ol{g(z)}$ for a locally univalent and sense-preserving log-harmonic mapping in $\mathbb{D}$, we mean that $m$ is a non-negative integer, $\text{Re}(\beta)>-1/2$, and $h$ and $g$ are analytic functions in $\mathbb{D}$ such that $g(0)=1$, $h(0)\not=0$ and $1^{2\beta m}=1$.
\subsection{Properties of the pre-Schwarzian and Schwarzian derivatives}
The following theorem characterizes the log-harmonic mapping with the analytic pre-Schwarzian derivative.
\begin{thm}\label{Theo2}
Let $f(z)=z^m|z|^{2\beta m} h(z)\ol{g(z)}$ be a locally univalent and sense-preserving log-harmonic mapping in $\mathbb{D}$. Then the pre-Schwarzian derivative $P_f$ of $f$ is analytic if, 
and only if, the second complex dilatation $\omega$ of $f$ is constant, {\it i.e.}, if, and only if, $f(z)=z^{(\beta+1) m} h(z)\;\ol{\gamma\left(z^{(\beta+1) m} h(z)\right)^a}$, where 
$a\in\mathbb{D}$ and $\gamma\in\mathbb{C}$.
\end{thm}
\begin{proof} The pre-Schwarzian derivative of $f$ is given in (\ref{c3}). It is evident that 
\beas \frac{\partial }{\partial \ol{z}}P_f=- \frac{\partial }{\partial \ol{z}}\left(\frac{\ol{\omega}\omega'}{1-|\omega|^2}\right)=\frac{-|\omega'|^2}{(1-|\omega|^2)^2}.\eeas
As $P_f$ of is analytic in $\mathbb{D}$, thus, we have $|\omega'(z)|\equiv 0$ in $\mathbb{D}$, {\it i.e.,} $\omega(z)\equiv \text{constant}$ in $\mathbb{D}$. Conversely, if $\omega(z)$ is constant in $\mathbb{D}$, then the pre-Schwarzian derivative of $f$ is given by 
\beas P_f&=&\frac{G'}{G}+\frac{H'}{H},~\text{where}~G=z^{(2\beta+1)m-1}g,~ H=zh'+(\beta+1)m h,\eeas
which is analytic in $\mathbb{D}$. Let $\omega\equiv a\in\mathbb{D}$. Then, we have  
\beas \frac{z g'(z)/g(z)+\beta m}{(\beta+1)m+z h'(z)/h(z)}=a,\quad{\it i.e.,}\quad \frac{g'(z)}{g(z)}+\frac{\beta m}{z}=a\left(\frac{(\beta+1)m}{z}+\frac{h'(z)}{h(z)}\right).\eeas
Thus, we have
\beas z^{\beta m}g(z)=\gamma\left(z^{(\beta+1)m} h(z)\right)^a,~\text{where}~\gamma\in\mathbb{C}.\eeas
Therefore, $f$ can be expressed as
\beas f(z)=z^{(\beta+1) m} h(z)\;\ol{\gamma\left(z^{(\beta+1) m} h(z)\right)^a},~\text{where}~ a\in\mathbb{D}~\text{and}~\gamma\in\mathbb{C}.\eeas
This completes the proof.
\end{proof}
The following theorem characterizes the log-harmonic mapping with the analytic Schwarzian derivative.
\begin{thm}\label{Theo3}
Let $f(z)=z^m|z|^{2\beta m} h(z)\ol{g(z)}$ be a locally univalent and sense-preserving log-harmonic mapping in $\mathbb{D}$. Then the Schwarzian derivative $S_f$ of $f$ is analytic if, 
and only if, the second complex dilatation $\omega$ of $f$ is constant, {\it i.e.}, if, and only if, $f(z)=z^{(\beta+1) m} h(z)\;\ol{\gamma\left(z^{(\beta+1) m} h(z)\right)^a}$, where 
$a\in\mathbb{D}$ and $\gamma\in\mathbb{C}$. 
\end{thm}
\begin{proof} 
Let $f(z)=z^m|z|^{2\beta m} h(z)\ol{g(z)}$ be a locally univalent and sense-preserving log-harmonic mapping in $\mathbb{D}$ that has an analytic Schwarzian derivative $S_f$ defined by 
\beas &&S_f=S_\phi-\frac{3}{2}\left(\frac{\ol{\omega}\omega'}{1-|\omega|^2}\right)^2+\frac{\ol{\omega}}{1-|\omega|^2}\left(\omega' \frac{\phi''}{\phi'}-\omega''\right), \eeas
\bea\text{\it i.e.,}&&(S_f- S_\phi)\left(1-|\omega|^2\right)^2+ \frac{3}{2}(\ol{\omega})^2(\omega')^2-\left(1-|\omega|^2\right)\ol{\omega}\left(\omega' \frac{\phi''}{\phi'}-\omega''\right)=0,\nonumber\\[1mm]\text{\it i.e.,}
&&\quad(S_f- S_\phi)+\ol{\omega}\left(\omega''-\omega' \frac{\phi''}{\phi'}-2(S_f- S_\phi)\omega\right)\nonumber\\[1mm]
\label{eq3}&&+(\ol{\omega})^2\left( \frac{3}{2}(\omega')^2+(S_f- S_\phi)\omega^2+\omega\left(\omega' \frac{\phi''}{\phi'}-\omega''\right)\right)=0,\eea
where $\phi(z)$ is an analytic and locally univalent function in $\mathbb{D}$ such that $\phi'(z)=(zh'+(\beta+1)m h)z^{(2\beta+1)m-1}g$. Suppose that $\omega$ is not constant in $\mathbb{D}$. Then there exists a disk $D(z_0, r):=\{z\in\mathbb{C}: |z-z_0|<r\}\subset \mathbb{D}$ such that $\omega'\not=0$ in $D(z_0, r)$, otherwise in view the Uniqueness theorem, we have $\omega$ is constant in $\mathbb{D}$. Differentiate (\ref{eq3}) with respect to $\ol{z}$, we obtain
\be\label{eq4}
\omega''-\omega' \frac{\phi''}{\phi'}-2(S_f- S_\phi)\omega+2\ol{\omega}\left( \frac{3}{2}(\omega')^2+(S_f- S_\phi)\omega^2+\omega\left(\omega' \frac{\phi''}{\phi'}-\omega''\right)\right)=0\ee
in $D(z_0,r)$. Differentiate (\ref{eq4}) with respect to $\ol{z}$, we obtain
\bea &&2\ol{\omega'}\left( \frac{3}{2}(\omega')^2+(S_f- S_\phi)\omega^2+\omega\left(\omega' \frac{\phi''}{\phi'}-\omega''\right)\right)=0,\nonumber\\[1mm]\text{\it i.e.,}
\label{eq5}&& \frac{3}{2}(\omega')^2+(S_f- S_\phi)\omega^2+\omega\left(\omega' \frac{\phi''}{\phi'}-\omega''\right)=0~\text{in}~D(z_0,r).\eea 
From (\ref{eq4}) and (\ref{eq3}), we have $\omega''-\omega' \phi''/\phi'-2(S_f- S_\phi)\omega=0$ and $S_f-S_\phi=0$ in $D(z_0, r)$, respectively. Thus, from (\ref{eq5}), we have $\omega'\equiv 0$ in $D(z_0, r)$, which contradicts the fact that $\omega$ is non-constant in $\mathbb{D}$. Hence, $\omega\equiv a\in\mathbb{D}$ and the
rest of the calculation follows from \textrm{Theorem \ref{Theo2}}. This completes the proof.
\end{proof}
The following theorem characterizes the log-harmonic mapping with the harmonic pre-Schwarzian derivative.
\begin{thm}\label{Theo4}
Let $f(z)=z^m|z|^{2\beta m} h(z)\ol{g(z)}$ be a locally univalent and sense-preserving log-harmonic mapping in $\mathbb{D}$. Then the pre-Schwarzian derivative $P_f$ of $f$ is harmonic if, 
and only if, the second complex dilatation $\omega$ of $f$ is constant.
\end{thm}
\begin{proof}
Let $f(z)=z^m|z|^{2\beta m} h(z)\ol{g(z)}$ be a locally univalent and sense-preserving log-harmonic mapping in $\mathbb{D}$ that has harmonic pre-Schwarzian derivative $P_f$ defined in (\ref{c3}). It is evident that 
\beas\frac{\partial^2 }{\partial z\partial \ol{z}}P_f=\frac{\partial }{\partial z}\left(\frac{-|\omega'|^2}{(1-|\omega|^2)^2}\right)
&=&\frac{-\ol{\omega'}\left((1-|\omega|^2)\omega''+2(\omega')^2\ol{\omega}\right)}{(1-|\omega|^2)^3}.\eeas 
As $P_f$ is harmonic, we have $\ol{\omega'}\left((1-|\omega|^2)\omega''+2(\omega')^2\ol{\omega}\right)\equiv 0$ in $\mathbb{D}$.
Suppose that $\omega$ is not constant in $\mathbb{D}$. Then, there exists a disk $D(z_0, r):=\{z\in\mathbb{C}: |z-z_0|<r\}\subset \mathbb{D}$ such that $\omega'\not=0$ in $D(z_0, r)$. Thus, we have 
\bea\label{eq6} (1-|\omega|^2)\omega''+2(\omega')^2\ol{\omega}\equiv 0,~\text{\it i.e.,}~\omega''+\left(2(\omega')^2-\omega\omega''\right)\ol{\omega}\equiv 0 ~\text{in}~D(z_0,r).\eea
Differentiate (\ref{eq6}) with respect to $\ol{z}$, we obtain
\bea\label{eq7}\left(2(\omega')^2-\omega\omega''\right)\ol{\omega'}\equiv 0,~\text{\it i.e.,}~2(\omega')^2-\omega\omega''\equiv 0 ~\text{in}~D(z_0,r).\eea
From (\ref{eq6}) and (\ref{eq7}), we have $\omega''\equiv 0$ in $D(z_0,r)$ and hence, from (\ref{eq7}), we have $\omega'\equiv 0$ in $D(z_0, r)$. This contradicts the fact that $\omega$ 
is not constant in $\mathbb{D}$. Thus, $\omega\equiv$ constant in $\mathbb{D}$. Conversely, if $\omega\equiv a\in\mathbb{D}$, then the pre-Schwarzian derivative of $f$ is given by $P_f=G'/G+H'/H$, which is analytic and hence, harmonic in $\mathbb{D}$, where $G=z^{(2\beta+1)m-1}g$, $H=zh'+(\beta+1)m h$. 
\end{proof}
The following theorem characterizes the log-harmonic mapping with the harmonic Schwarzian derivative.
\begin{thm}\label{Theo5}
Let $f(z)=z^m|z|^{2\beta m} h(z)\ol{g(z)}$ be a locally univalent and sense-preserving log-harmonic mapping in $\mathbb{D}$. Then the Schwarzian derivative $S_f$ of $f$ is harmonic if, 
and only if, the second complex dilatation $\omega$ of $f$ is constant.
\end{thm}
\begin{proof} 
Let $f(z)=z^m|z|^{2\beta m} h(z)\ol{g(z)}$ be a locally univalent and sense-preserving log-harmonic mapping in $\mathbb{D}$ that has harmonic Schwarzian derivative $S_f$ defined in (\ref{c4}). It is evident that 
\beas\frac{\partial^2 }{\partial z\partial \ol{z}}S_f&=&\frac{\partial }{\partial z}\left(-\frac{3(\omega')^2\ol{\omega}\ol{\omega'}}{(1-|\omega|^2)^3}+\frac{\left(\omega' P_\phi-\omega''\right)\ol{\omega' }}{(1-|\omega|^2)^2}\right)\\[1mm]
&=&\frac{-6\left(1-|\omega|^2\right)\omega'\omega''\ol{\omega}\ol{\omega'}-9(\omega')^3(\ol{\omega})^2\ol{\omega'}}{(1-|\omega|^2)^4}\\[2mm]
&&+\frac{(1-|\omega|^2)^2\left(\omega' P_\phi-\omega''\right)'\ol{\omega' }+2(1-|\omega|^2)\ol{\omega' }\omega'\ol{\omega}\left(\omega' P_\phi-\omega''\right)}{(1-|\omega|^2)^4},\eeas
where $\phi(z)$ is an analytic and locally univalent function in $\mathbb{D}$ such that $\phi'(z)=(zh'+(\beta+1)m h)z^{(2\beta+1)m-1}g$.
As $S_f$ is harmonic, we have 
\bea&& -6\left(1-|\omega|^2\right)\omega'\omega''\ol{\omega}\ol{\omega'}-9(\omega')^3(\ol{\omega})^2\ol{\omega'}+(1-|\omega|^2)^2\left(\omega' P_\phi-\omega''\right)'\ol{\omega' }\nonumber\\[2mm]
&&+2(1-|\omega|^2)\ol{\omega' }\omega'\ol{\omega}\left(\omega' P_\phi-\omega''\right)\equiv 0,\nonumber\\[1mm]\text{\it i.e.,}
&&2\left(\omega'\left(\omega' P_\phi-\omega''\right)-3\omega'\omega''\right)\left(1-|\omega|^2\right)\ol{\omega}\ol{\omega'}-9(\omega')^3(\ol{\omega})^2\ol{\omega'}\nonumber\\[2mm]
\label{eq8}&&+\left(\omega' P_\phi-\omega''\right)'(1-|\omega|^2)^2\ol{\omega' }\equiv 0\quad\text{in}\quad\mathbb{D}.\eea
Suppose that $\omega$ is not constant in $\mathbb{D}$. Then there exists a disk $D(z_0, r):=\{z\in\mathbb{C}: |z-z_0|<r\}\subset \mathbb{D}$ such that $\omega'\not=0$ in $D(z_0, r)$. From (\ref{eq8}), we have 
\bea &&2\left(\omega'\left(\omega' P_\phi-\omega''\right)-3\omega'\omega''\right)\left(1-|\omega|^2\right)\ol{\omega}-9(\omega')^3(\ol{\omega})^2\nonumber\\[2mm]
&&+\left(\omega' P_\phi-\omega''\right)'(1-|\omega|^2)^2\equiv 0,\nonumber\\[1mm]\text{\it i.e.,}
&&\left(\omega' P_\phi-\omega''\right)'+\ol{\omega}\left(2\left(\omega'\left(\omega' P_\phi-\omega''\right)-3\omega'\omega''\right)-2\left(\omega' P_\phi-\omega''\right)'\omega\right)\nonumber\\
\label{eq9}&&+(\ol{\omega})^2\left(\left(\omega' P_\phi-\omega''\right)'\omega^2-2\left(\omega'\left(\omega' P_\phi-\omega''\right)-3\omega'\omega''\right)\omega-9(\omega')^3\right)\equiv 0\nonumber\\\eea
in $D(z_0, r)$. Differentiate (\ref{eq9}) with respect to $\ol{z}$, we obtain
\bea\label{eq10} &&\left(\left(\omega'\left(\omega' P_\phi-\omega''\right)-3\omega'\omega''\right)-\left(\omega' P_\phi-\omega''\right)'\omega\right)+\ol{\omega}\left(\left(\omega' P_\phi-\omega''\right)'\omega^2\right.\nonumber\\
&&\left.-2\left(\omega'\left(\omega' P_\phi-\omega''\right)-3\omega'\omega''\right)\omega-9(\omega')^3\right)\equiv 0\eea
in $D(z_0, r)$. Differentiate (\ref{eq10}) with respect to $\ol{z}$, we obtain
\bea\label{eq11}\left(\omega' P_\phi-\omega''\right)'\omega^2-2\left(\omega'\left(\omega' P_\phi-\omega''\right)-3\omega'\omega''\right)\omega-9(\omega')^3\equiv 0\eea
in $D(z_0, r)$. Using (\ref{eq11}) in (\ref{eq10}) and (\ref{eq9}), we have, respectively 
\bea\label{eq12}&&\left(\omega'\left(\omega' P_\phi-\omega''\right)-3\omega'\omega''\right)-\left(\omega' P_\phi-\omega''\right)'\omega\equiv 0\\\text{and}
\label{eq13}&&\left(\omega' P_\phi-\omega''\right)'\equiv 0\quad\text{in}\quad D(z_0,r).\eea
From (\ref{eq12}) and (\ref{eq13}), we have $\left(\omega' P_\phi-\omega''\right)-3\omega''\equiv 0$ in $D(z_0,r)$. Thus, from (\ref{eq11}), we have $\omega'\equiv 0$ in 
$D(z_0,r)$, which contradicts the fact that $\omega$ is not constant in $\mathbb{D}$. 
Thus, $\omega\equiv$ constant in $\mathbb{D}$. Conversely, if $\omega\equiv a\in\mathbb{D}$, then the Schwarzian derivative of $f$ is given by 
$S_f=S_\phi$, which is analytic and hence, harmonic in $\mathbb{D}$. This completes the proof.
\end{proof}
\begin{thm}
Let $f(z)=z^m|z|^{2\beta m} h(z)\ol{g(z)}$ be a locally univalent and sense-preserving log-harmonic mapping in $\mathbb{D}$ and $\phi(z)$ be an analytic and locally univalent function in $\mathbb{D}$ such that $\phi'(z)=(zh'+(\beta+1)m h)z^{(2\beta+1)m-1}g$.  If 
\bea\label{r2} |z P_{f}|+\frac{|z \omega'(z)|}{1-|\omega(z)|^2}\leq \frac{1}{(1-|z|^2)},\eea
 then $\phi$ is univalent in $\mathbb{D}$.
\end{thm}
\begin{proof} The pre-Schwarzian derivative of $f$ is given by $P_f=P_\phi-\ol{\omega}\omega'/\left(1-|\omega|^2\right)$,
where $\omega$ is the second complex dilatation of $f$, which is also a Schwarz function in $\mathbb{D}$. Using (\ref{r2}), we have 
\beas \left|z P_\phi\right|=\left|z P_f\right|+\left|\frac{z\ol{\omega}\omega'}{1-|\omega|^2}\right|\leq \left|z P_f\right|+\frac{\left|z\omega'\right|}{1-|\omega|^2}\leq \frac{1}{1-|z|^2}.\eeas
Therefore, we have $\sup_{z\in\mathbb{D}}|zP_\phi|\leq 1$. In view of the Becker's univalence criterion, $\phi$ is univalent in $\mathbb{D}$. This completes the proof.
\end{proof}
\noindent{\bf Declarations}\\
\noindent{\bf Acknowledgment:} The work of the first author is supported by University Grants Commission (IN) fellowship (No. F. 44 - 1/2018 (SA - III)).\\
{\bf Conflict of Interest:} The authors declare that there are no conflicts of interest regarding the publication of this paper.\\
{\bf Availability of data and materials:} Not applicable.


\begin{thebibliography}{99}


\bibitem{A1996} {\sc Z. Abdulhadi}, Close-to-starlike logharmonic mappings, {\it Internat. J. Math. \& Math. Sci.}, {\bf 19}(3) (1996), 563--574.

\bibitem{AB1988}{\sc Z. Abdulhadi} and {\sc D. Bshouty}, Univalent functions in $H\cdot\overline{H}(D)$, {\it Tran. Amer. Math. Soc.}, {\bf 305}(2) (1988), 841--849.

\bibitem{AH1987} {\sc Z. Abdulhadi} and {\sc W. Hengartner}, Spirallike logharmonic mappings, {\it Complex Var. Theory Appl.}, {\bf 9}(2-3) (1987), 121--130.

\bibitem{AA2012}{\sc Z. Abdulhadi} and {\sc R.M. Ali}, Univalent logharmonic mappings in the plane, {\it Abstr. Appl. Anal.}, {\bf 2012} (2012), Art. ID 721943, 32 pp.

\bibitem{AP2024}{\sc M. F. Ali} and {\sc S. Pandit}, On the pre-Schwarzian norm of certain logharmonic mappings, {\it Bull. Malays. Math. Sci. Soc.}, {\bf47} (2024), 67.

\bibitem{ACP1974}
{\sc J. Anderson, J. Clunie} and {\sc C. Pommerenke}, On Bloch functions and normal functions, Walter de Gruyter, Berlin/New York Berlin, New York, (1974), 12--37.

\bibitem{B1972}
{\sc J. Becker}, L\"{o}wnersche Differentialgleichung und quasikonform fortsetzbare schlichte Funktionen, {\it J. Reine Angew. Math.}, {\bf 255} (1972), 23--43.

\bibitem{BP1984}
{\sc J. Becker} and {\sc C. Pommerenke}, Schlichtheitskriterien und Jordangebiete, {\it J. Reine Angew. Math.}, {\bf 354} (1984), 74--94.

\bibitem{BHPV2022}
{\sc V. Bravo, R. Hern{\'a}ndez, S. Ponnusamy} and {\sc O. Venegas}, Pre-Schwarzian and Schwarzian derivatives of logharmonic mappings, {\it Monatsh. Math.}, {\bf199} (2022),  733--754.

\bibitem{CDO2003}
{\sc M. Chuaqui, P. Duren} and {\sc B. Osgood}, The Schwarzian derivative for harmonic mappings, {\it J. Anal. Math.}, {\bf 91}(1) (2003), 329--351.

\bibitem{D1983}{\sc P. L. Duren}, Univalent Functions, Springer-Verlag, 1983.
\bibitem{D2004} {\sc P. Duren}, Harmonic mapping in the plane, {\it Cambridge University Press}, 2004.


\bibitem{HM2015}{\sc R. Hern{\'a}ndez, M. J. Mart\'in},  Pre-Schwarzian and Schwarzian derivatives of harmonic mappings, {\it J. Geomet. Anal.}, {\bf 25}(1)  (2015), 64--91.
\bibitem{H1949}{\sc E. Hille}, Remarks on a paper by Zeev Nehari, {\it Bull. Am. Math. Soc.} {\bf55} (1949), 552--553.

\bibitem{K1932}
{\sc W. Kraus}, Uber den Zusammenhang eigner Characterstiken eines einfach zusammenhangenden Bereiches mit der Kreisabbildung, {\it Mitt. Math. Sem. Giessen}, {\bf 21}  (1932),  1--28.
\bibitem{KS2002} {\sc Y. C. Kim} and {\sc T. Sugawa}, Growth and coefficient estimates for uniformly locally univalent functions on the unit disk, {\it Rocky Mountain J. Math.}, {\bf 32}  (2002), 179--200.
\bibitem{L1936} {\sc H. Lewy}, On the non-vanishing of the Jacobian in certain one-to-one mappings, {\it Bull. Am. Math. Soc.}, 42 (1936), 689--692.

\bibitem{1LP2018}
{\sc Z. Liu} and {\sc S. Ponnusamy}, Some properties of univalent log-harmonic mappings, {\it Filomat}, {\bf 32}(15) (2018), 5275--5288.

\bibitem{2LP2018}{\sc G. Liu} and {\sc S. Ponnusamy}, Uniformly locally univalent harmonic mappings associated with the pre-Schwarzian norm, {\it 	Indag. Math.}, {\bf 29}(2) (2018), 752--778.
\bibitem{LP2022}
{\sc Zh. Liu} and {\sc S. Ponnusamy}, On univalent log-harmonic mappings, {\it Filomat}, {\bf 36}(12) (2022), 4211--4224.
\bibitem{MPW2013}
{\sc Z Mao, S. Ponnusamy} and {\sc X. Wang}, Schwarzian derivative and Landau's theorem for log-harmonic mappings, {\it Complex Var. Elliptic Equ.}, {\bf 58}(8) (2013), 1093--1107.

\bibitem{N1949}{\sc Z. Nehari}, The Schwarzian derivative and schlicht functions, {\it Bull. Amer. Math. Soc.}, {\bf 55}(6) (1949), 545--551.
\bibitem{N1989}{\sc J. C. C. Nitsche}, Lectures on minimal surfaces: Volume $1$, {\it Cambridge University Press}, NewYork, 1989.

\bibitem{P1970}
{\sc C. Pommerenke}, On Bloch functions, {\it J. London Math. Soc.}, {\bf 2}(2) (1970), 689--695.

\bibitem{Y1976}
{\sc S. Yamashita}, Almost locally univalent functions, \textit{Monatsh. Math.}, {\bf 81} (1976), 235--240.

\end{thebibliography}
\end{document}